\documentclass[a4paper]{article}
\usepackage[utf8]{inputenc}

\usepackage[top=2.5cm, bottom=2.5cm, left=2.5cm, right=2.5cm]{geometry}

\usepackage{amsmath, amssymb, comment, xcolor}
\usepackage[hidelinks]{hyperref}

 \usepackage{amsthm}
 \usepackage[noabbrev]{cleveref}
 \newtheorem{thm}{Theorem}[section]
 \newtheorem{lm}[thm]{Lemma}
 \newtheorem{res}[thm]{Result}
 \newtheorem{crl}[thm]{Corollary}
 \newtheorem{prop}[thm]{Proposition}

 \theoremstyle{definition}
 
 \newtheorem{rmk}[thm]{Remark}
 \newtheorem{df}[thm]{Definition}
 \newtheorem{constr}[thm]{Construction}
 \newtheorem{prob}{Problem}

 \newcommand{\FF}{\mathbb F}

 \newcommand{\vspan}[1]{\left \langle #1 \right \rangle}
 \newcommand{\set}[1]{ \left \{ #1 \right \} }
 \newcommand{\sett}[2]{ \left\{ #1 \, \, || \, \, #2 \right \} }

 \newcommand{\gauss}[2]{\genfrac{[}{]}{0pt}{}{#1}{#2}_q}
 \newcommand{\pg}{\textnormal{PG}}
 \newcommand{\PG}{\textnormal{PG}}
 
 \newcommand{\mc}{\mathcal C}
 \newcommand{\wt}{\textnormal{wt}}
 \newcommand{\supp}{\textnormal{supp}}
 
 \newcommand{\one}{\mathbf 1}
 \newcommand{\zero}{\mathbf 0}

 \newcommand{\mm}{\mathcal M}
 
 \renewcommand{\mp}{\mathcal P}
 \newcommand{\mb}{\mathcal B}
 
 \newcommand{\mo}{\mathcal O}

 \renewcommand{\phi}{\varphi}

\title{
A note on small weight codewords of projective geometric codes and on the smallest sets of even type
}
\author{Sam Adriaensen\thanks{Pleinlaan 2, 1050 Elsene, Belgium. \url{sam.adriaensen@vub.be}.} \\ {\it Vrije Universiteit Brussel} }
\date{}

\begin{document}

\maketitle

\begin{abstract}
 In this paper, we study the codes $\mc_k(n,q)$ arising from the incidence of points and $k$-spaces in $\PG(n,q)$ over the field $\FF_p$, with $q = p^h$, $p$ prime.
 We classify all codewords of minimum weight of the dual code $\mc_k(n,q)^\perp$ in case $q \in \{4,8\}$.
 This is equivalent to classifying the smallest sets of even type in $\pg(n,q)$ for $q \in \{4,8\}$.
 We also provide shorter proofs for some already known results, namely of the best known lower bound on the minimum weight of $\mc_k(n,q)^\perp$ for general values of $q$, and of the classification of all codewords of $\mc_{n-1}(n,q)$ of weight up to $2q^{n-1}$.
\end{abstract}



\section{Introduction}

In this paper, we study codes arising from the incidence relation of points and subspaces of a fixed dimension $k$ in projective space $\pg(n,q)$, with $q=p^h$, $p$ prime.
Let $\mp$ denote the set of points of $\PG(n,q)$.
Then the code $\mc_k(n,q)$ is defined as the subspace of the vector space $\FF_p^\mp$ of functions $\mp \to \FF_p$, spanned by the characteristic vectors of the $k$-spaces.

The study of these codes was initiated over 50 years ago, see e.g.\ Goethals and Delsarte \cite{goethals}.
Since then, there has been a great interest in small weight codewords both in $\mc_k(n,q)$ and its dual code $\mc_k(n,q)^\perp$.
The most challenging problem is determining the minimum weight of the dual codes, and classifying all codewords of minimum weight.
Lavrauw, Storme, and Van de Voorde \cite{lavrauwB} reduced this problem to the case $k=1$.
So far, the minimum weight of $\mc_1(n,q)$ has only been determined exactly when $q$ is prime, when $q$ is even, and for some small values of $n$ and $q$.
Furthermore, in general, we only have a classification of minimum weight codewords of $\mc_1(n,q)^\perp$ in case $q$ is prime.

If $q$ is even, then codewords of $\mc_1(n,q)^\perp$ of weight $w$ are equivalent to sets of points of size $w$ in $\pg(n,q)$ that intersect every line in an even number of points.
Such sets are called \emph{sets of even type}.
Calkin, Key, and de Resmini \cite{calkinkeyderesmini} used a BCH bound by Delsarte \cite{delsarte70bch} to prove that the minimum weight of $\mc_1(n,q)^\perp$, or equivalently the minimum size of a set of even type in $\pg(n,q)$, is $q^{n-2}(q+2)$.
The only known examples of sets of even type of this size are cylinders with an $(n-3)$-dimensional vertex over hyperovals.
It is known that these are the only examples in case $q=2$ \cite[Proposition 3]{bagchi} or $q=4$ and $n\in \{3,4\}$ \cite{hirschfeldhubaut}, \cite[Table 4.9]{packer95}.
The main result of this paper is the following.

\begin{thm}
 \label{ThmMainHypercylinders}
 If $q \in \{4,8\}$, the only sets of even type in $\pg(n,q)$ of size $q^{n-2}(q+2)$ are cylinders over regular hyperovals.
\end{thm}
As discussed above, this extends the classification of the minimum weight codewords of $\mc_k(n,q)^\perp$ for $q=2$ to the cases $q \in \{4,8\}$.

\bigskip

The rest of this paper is concerned with providing simpler proofs for known results in the study of these codes.
We give a short proof of the best-known general lower bound on the minimum weight of $\mc_k(n,q)^\perp$, which was originally proven by Bagchi and Inamdar \cite{bagchi}.
We remark that this proof was previously discovered by Aart Blokhuis, but never published \cite{zsuzsa}.
Recently, similar ideas have been used in \cite{deboeckvandevoorde2022} to obtain a small improvement on the bound by Bagchi and Inamadar for the codes $\mc_1(2,p^2)$, $p$ prime.

\bigskip

The classification of codewords of small weight of $\mc_k(n,q)$ has also received a great deal of attention.
Building on the work of Lavrauw, Storme, Sziklai, and Van de Voorde \cite{lavrauw,lavrauwstormesziklaivandevoorde}, Polverino and Zullo \cite{polverino} proved the following result.

\begin{res}[{\cite{polverino}}]
 \label{ResOlgaFerdi}
 Suppose that $c \in \mc_{n-1}(n,q)$ is a non-zero codeword of weight $w \leq 2q^{n-1}$.
 Then either
 \begin{itemize}
  \item $w = \frac{q^n-1}{q-1}$ and $c$ is the scalar multiple of the characteristic function of a hyperplane, or
  \item $w=2q^{n-1}$ and $c$ is the scalar multiple of the difference of two characteristic functions of hyperplanes.
 \end{itemize}
\end{res}
In this paper, we provide a significantly shorter, self-contained proof of this result.
This proof uses a generalisation of a lemma by Blokhuis, Brouwer, and Wilbrink \cite{blokhuisbrouwerwilbrink}, which was discovered by Olga Polverino and Ferdinando Zullo, but never published.

We remark that for large values of $q$, the above classification of codewords of $\mc_{n-1}(n,q)$ has been extended to codewords of weight up to roughly $4q^{n-1}$ in \cite{adriaensendenauxstormeweiner} and codewords of weight up to roughly $\frac 1 {2^{n-2 (1-\delta_{h,2})}} \sqrt q q^{n-1}$ for $q=p^h$ with $h>1$ in \cite{bartolidenaux}, building on results by Sz\H onyi and Weiner \cite{szonyiB}.
The classification was generalised to the codes $\mc_k(n,q)$ in \cite{adriaensendenaux} for codewords of weight up to roughly $3q^k$.

\bigskip

\noindent \textbf{Structure of the paper.}
\Cref{SecPrel} contains definitions and notation followed throughout the article.
In \Cref{SecMinWtDual} we first provide a simple proof of the bound by Bagchi and Inamdar (without their characterisation in case of equality) on the minimum weight of dual codes of incidence structures.
Afterwards, we prove \Cref{ThmMainHypercylinders}, and provide some intermediary results that could be helpful in classifying the smallest sets of even type for larger values of $q$.
\Cref{SecSmallWtPtHyp} contains the new proof of \Cref{ResOlgaFerdi}.
Finally, \Cref{SecConclusion} contains some concluding remarks.

\section{Preliminaries}
 \label{SecPrel}

Throughout this article, $p$ will denote a prime number and $q=p^h$ for some positive integer $h$.
The finite field of order $q$ will be denoted as $\FF_q$.
The $n$-dimensional projective space over $\FF_q$ will be denoted as $\pg(n,q)$.
We use the notation
\[
 \theta_n = \frac{q^{n+1}-1}{q-1} = q^n + q^{n-1} + \ldots + q + 1
\]
for the number of points in $\pg(n,q)$.

\bigskip

We will describe linear codes arising from incidence structures in their general form.
An \emph{incidence structure} is a pair $(\mp,\mb)$, with $\mb$ a (multi)set of subsets of $\mp$.
The elements of $\mp$ and $\mb$ are called \emph{points} and \emph{blocks} respectively.
For every block $B \in \mb$, we define its \emph{characteristic function} as
\[
 \chi_B: \mp \to \{0,1\}: P \mapsto \begin{cases}
  1 & \text{if } P \in B, \\
  0 & \text{otherwise.}
 \end{cases}
\]
Given a field $\FF$, we can consider these characteristic functions as functions $\mp \to \FF$, since $\FF$ contains a 0 and a 1.
We will denote the $\FF$-vector space of all functions $\mp \to \FF$ as $\FF^\mp$.
We will denote the constant functions mapping everything to 0 and 1 respectively by $\zero$ and $\one$.

\begin{df}
 Let $D = (\mp,\mb)$ be an incidence structure, and $\FF$ a field.
 The \emph{code} of $D$ over $\FF$ is defined as the subspace of $\FF^\mp$ spanned by the characteristic vectors of the blocks of $D$.
\end{df}

Alternatively, these codes can be defined as follows.
Choose an ordering for the points and blocks of $D$, i.e.\ suppose that $\mp = \{P_1, \dots, P_m\}$ and $\mb = \{B_1, \dots, B_n\}$.
Then we can define the incidence matrix of $D$ as the $n \times m$-matrix $M = (m_{i,j})$ with
\[
 m_{i,j} = \begin{cases}
  1 & \text{if } P_j \in B_i, \\
  0 & \text{otherwise}.
 \end{cases}
\]
We can define the code of $D$ over $\FF$ as the rowspace of the incidence matrix $M$.
Remark that changin the order of the elements of $\mb$ has no impact on the code, and reordering the elements of $\mp$ permutes the coordinate positions of the code, which gives rise to an equivalent code.
Hence, the ordering of the points and blocks makes no meaningful difference in the definition of the code.

If $c$ is a function in $\FF^\mp$, its \emph{support} is $\supp(c) = \sett{P \in \mp}{c(P) \neq 0}$ and its \emph{weight} is $\wt(c) = |\supp(c)|$.
The \emph{minimum weight} of a subspace $C \leq \FF^\mp$ is $d(C) = \min_{c \in C \setminus \set \zero} \wt(c)$.

The standard scalar product on $\FF^\mp$ is the non-degenerate symmetric bilinear form given by
\[
 v \cdot w = \sum_{P \in \mp} v(P) w(P).
\]
If $C$ is a subspace of $\FF^\mp$, its \emph{orthogonal complement} is the subspace
\[
 C^\perp = \sett{c \in \FF^\mp}{(\forall v \in C)(c \cdot v = 0)}.
\]
Subspaces of $\FF^\mp$ are also called \emph{codes}, and in that case, we call $C^\perp$ the \emph{dual code} of $C$.
Note that if $C$ is the code of an incidence structure $D = (\mp,\mb)$ over $\FF$, then
\[
 C^\perp = \sett{c \in \FF^\mp}{(\forall B \in \mb)( c \cdot \chi_B = 0 )}.
\]

\begin{df}
 Suppose that $q=p^h$, $p$ prime.
 Let $D$ be the incidence structure of points and $k$-spaces in $\pg(n,q)$.
 Then $\mc_k(n,q)$ will denote the code of $D$ over $\FF_p$.
\end{df}

For a survey on results concerning $\mc_k(n,q)$ and its dual, we refer the reader to \cite{lavrauwC}.
This survey dates from 2010, so since it publications, there have been advances in the study of these codes, see e.g.\ the papers mentioned in the introduction.

\begin{rmk}
 One could wonder why we only study the codes of points and $k$-spaces in $\pg(n,q)$ over $\FF_p$, and not over another finite field $\FF_{q'}$.
 This is well-motivated.
 Since the code is generated by functions taking only the values 0 and 1, the dimension and minimum weight of the code over $\FF_{q'}$ only depend on the characteristic of $\FF_{q'}$, so it makes sense to study these codes over the smallest field of given characteristic.
 Moreover, if the characteristic of $\FF_{q'}$ is not $p$, the dual code of points and $k$-spaces of $\pg(n,q)$ over $\FF_{q'}$ is either $\{\zero\}$ or $\vspan \one$ \cite[Theorem 2.5]{lavrauwC}.
\end{rmk}

We conclude with one more notational convention.
For two sets $A$ and $B$, we denote the symmetric difference of $A$ and $B$ as $A \triangle B = (A \cup B) \setminus (A \cap B)$.

\section{Minimum weight of the dual code}
 \label{SecMinWtDual}

In this section, we will investigate the minimum weight codewords of $\mc_k(n,q)^\perp$.
We first present a result by Lavrauw, Storme, and Van de Voorde that reduces the study of minimum weight codewords of $\mc_k(n,q)^\perp$ to the case $k=1$.

\bigskip

Suppose that $\pi$ is a subspace of $\pg(n,q)$, and let $\mp_\pi$ and $\mp$ denote the sets of points of $\pi$ and $\pg(n,q)$ respectively.
Then we can consider $\FF_p^{\mp_\pi}$ as a subspace of $\FF_p^{\mp}$.
Formally, for each $c \in \FF_p^{\mp_\pi}$, we can define $\iota(c) \in \FF_p^\mp$ as
\[
 \iota(c): \mp \to \FF_p: P \mapsto \begin{cases}
  c(P) &\text{if } P \in \pi, \\
  0 & \text{otherwise}.
 \end{cases}
\]
Then $\iota$ is a linear embedding of $\FF^{\mp_\pi}$ into $\FF^\mp$.
Let $\mc_k(\pi)$ denote the code $\mc_k(\dim \pi,q)$ defined on the points of $\pi$.
Then we can regard $\mc_k(\pi)$ and $\mc_k(\pi)^\perp$ as subspaces of $\FF_p^\mp$.
This allows us to rephrase the result by Lavrauw, Storme, and Van de Voorde in the following way.

\begin{res}[{\cite[Theorem 11]{lavrauwB}}]
 \label{ResReductionToLines}
 A codeword $c \in \mc_k(n,q)^\perp$ is of minimum weight if and only if there exists an $(n-k+1)$-space $\pi$ of $\pg(n,q)$ such that $c$ is a minimum weight codeword of $\mc_1(\pi)^\perp$.
\end{res}

\subsection{Lower bound on the minimum weight of the dual code}

The best general bound on the minimum weight of $\mc_1(n,q)^\perp$ known to the author follows from a theorem by Bagchi and Inamdar.

\begin{res}[{\cite[Theorem 3]{bagchi}}]
\label{ResLowerBound}
Let $D = (\mp,\mb)$ be an incidence structure, where every pair of points lies in at most $\lambda$ blocks, and every point lies in at least $n+\lambda$ blocks.
Consider the code $C$ of $D$ over a prime field $\FF_p$.
Take a non-zero codeword $c \in C^\perp$.
\begin{enumerate}
 \item Then $\wt(c) \geq 2 \left( \frac{n + \lambda}{\lambda} - \frac{n}{\lambda p} \right)$.
 \item If equality holds in (1), then there exists a scalar $\alpha \in \FF_p^*$ such that
\begin{itemize}
 \item the image of $c$ is $\{0,\pm \alpha\}$ and there are equally many points with coefficient $\alpha$ as points with coefficient $-\alpha$.
 \item any block intersects $c$ in either $p$ points with the same coefficient $\alpha$ or $-\alpha$, in two points with coefficients $\alpha$ and $-\alpha$, or in the empty set.
 \item any two points of $\supp(c)$ lie in exactly $\lambda$ blocks.
\end{itemize}
\end{enumerate}
\end{res} 

While their proof is rather long, we present a short proof for point (1).

\bigskip

\noindent \textit{Proof of Result \ref{ResLowerBound} (1).}
Each element $\alpha$ of $\FF_p$ corresponds naturally to an integer in $\set{0, \dots, p-1}$, which we will denote by $\nu(\alpha)$.
Consider for each $c \in C^\perp$ the multiset $\mm(c)$, where the elements of $\mm(c)$ are points of $\mp$, and each point $P$ has multiplicity $\nu(c(x))$ in $\mm(c)$.
Now take a minimum-weight codeword $c \in C^\perp$.
By rescaling $c$ we may assume that $c(P) = 1$ for some point $P \in \mp$.
There are at least $n+\lambda$ blocks through $P$.
Since $c \in C^\perp$, they must all contain at least $p-1$ other points of $\mm(c)$ (counted with multiplicity).
Every other point is counted at most $\lambda$ times.
This yields $$|\mm(c)| \geq 1 + \frac{n + \lambda}{\lambda}(p-1) = \frac{n + \lambda}{\lambda}p - \frac{n}{\lambda}.$$
Now we distinguish two cases.
First, suppose that there is no point $Q \in \mp$ with $c(Q) = -1$.
Then every block through $P$ contains at least 2 more points of $\supp(c)$, and $\wt(c) \geq 2 \frac{n+\lambda}{\lambda} + 1$.

Now suppose that there does exist a point $Q \in \mp$ with $c(Q) = -1$.
Then $(-c)(Q) = 1$, and we can apply the same reasoning as before to prove that
$$|\mm(-c)| \geq \frac{n + \lambda}{\lambda}p - \frac{n}{\lambda}.$$
Note that $|\mm(c)| + |\mm(-c)| = p \wt(c) $, because every point $R$ of $\supp(c)$ with $c(R) = \alpha \neq 0$ contributes $\nu(\alpha)$ to $|\mm(c)|$ and $\nu(-\alpha) = p - \nu(\alpha)$ to $|\mm(-c)|$.
This yields
\begin{equation}
 \tag*{\qedsymbol}
 \wt(c) = \frac{|\mm(c)| + |\mm(-c)|} p \geq \frac 2 p \left( \frac{n + \lambda}{\lambda}p - \frac{n}{\lambda} \right).
\end{equation}

We can draw some weaker conclusions in case of equality.
In that case, if $c$ is a minimum weight codeword then every point of $\supp(c)$ lies on exactly $n+\lambda$ blocks and any pair of points of $\supp(c)$ lies on exactly $\lambda$ blocks.
Furthermore, if $c(P) = \alpha \neq 0$, then every block through $P$ intersects $\mm(\alpha^{-1} c)$ in exactly $p$ points (counted with multiplicity). \Cref{ResLowerBound} (2) would follow if one can prove that $c$ only takes values 0 and $\pm \alpha$ for some $\alpha \notin \FF_p^*$.
However, it is not obvious how to derive this property from the arguments presented above for general values of $p$.

\begin{crl}[{\cite[Theorem 3]{bagchi}}]
 \label{CrlMinWeight}
 \[
  d(\mc_1(n,q)^\perp) \geq 2 \left( \theta_{n-1} \left( 1 - \frac 1p \right) + \frac 1p \right).
 \]
\end{crl}

\begin{rmk}
It was privately communicated to the author by Zsuzsa Weiner \cite{zsuzsa} that Aart Blokhuis independently discovered the same proof as above, but never published this.
Recently, similar arguments have been used by De Boeck and Van de Voorde \cite{deboeckvandevoorde2022} to increase the lower bound on the minimum weight of $\mc_1(2,p^2)^\perp$, $p \geq 5$ prime, from $2q-2p+2$ to $2q-2p+5$.
\end{rmk}

\subsection{Minimum weight of the dual code for \texorpdfstring{$q$}{\textit{q}} even}

For general $n$, the minimum weight of $\mc_1(n,q)^\perp$ is only known when $q$ is prime or even.
Moreover, for general $n$, minimum weight codewords are only characterised for $q$ prime.
In this subsection, we will investigate codewords of minimum weight of $\mc_1(n,q)$ for $q$ even.
We will characterise these codewords for $q \in \{4,8\}$, and state the intermediary results as generally as possible, in the hope they might someday be helpful to obtain a characterisation for higher values of $q$.

\bigskip

Throughout this section, assume that $q$ is even.
In this case a codeword $c$ only takes values in $\FF_2$.
Let $\mp$ denote the set of points of $\pg(n,q)$.
Then there is a 1-1 correspondence between the functions $c \in \FF_2^\mp$ of weight $w$, and sets of points of $\pg(n,q)$ of size $w$.
This correspondence is given by going from a function to its support, and from a set of points to its characteristic function.
Moreover, $c \in \mc_1(n,q)^\perp$ if and only if $\supp(c)$ is a set which intersects every line of $\pg(n,q)$ in an even number of points.
Call such a set a \emph{set of even type}.
Characterising minimum weight codewords of $\mc_1(n,q)^\perp$ is equivalent to characterising the smallest sets of even type in $\pg(n,q)$.
The minimum weight of $\mc_1(n,q)$, hence the minimum size of sets of even type in $\pg(n,q)$, is known.
It was derived by Calkin, Key, and de Resmini \cite{calkinkeyderesmini} from a BCH bound established by Delsarte \cite{delsarte70bch}.

\begin{res}[{\cite{calkinkeyderesmini,delsarte70bch}}]
 \label{ResSizeSmallestEvenSet}
Let $q$ be even, $n \geq 2$.
The smallest sets of even type in $\pg(n,q)$ have size $(q+2)q^{n-2}$.
\end{res}

We will call sets of even type of size $q^{n-2}(q+2)$ in $\pg(n,q)$ \emph{minimum sets of even type }.
The previous result tells us that the minimum sets of even type in $\pg(2,q)$ have size $q+2$.
This is a well-known result in finite geometry.
Such sets are called \emph{hyperovals}, and they intersect every line of $\pg(2,q)$ in 0 or 2 points.
They have been classified for $q \leq 64$ \cite{vandendriessche19}.
Segre \cite{segre57} proved already in the fifties that if $q \leq 8$, all hyperovals are \emph{regular}, i.e.\ projectively equivalent to
\[
 \sett{(s^2,st,t^2)}{(s,t) \in \pg(1,q)} \cup \set{(0,1,0)}.
\]

To prove that the bound in \Cref{ResSizeSmallestEvenSet} is tight, Calkin, Key, and de Resmini provided the following construction of minimum sets of even type in $\pg(n,q)$.

\begin{constr}[{\cite[Corollary 1]{calkinkeyderesmini}}]
 \label{ResHypercylinder}
 Let $\pi$ and $\tau$ be skew subspaces $\pg(n,q)$ of dimensions 2 and $n-3$ respectively.
 Let $\mo$ be a hyperoval in $\pi$.
 Then the set
 \[
  \left( \bigcup_{P \in \mo} \vspan{P,\tau} \right) \setminus \tau.
 \]
 is a minimum set of even type in $\pg(n,q)$.
\end{constr}

We adopt the terminology from \cite{adriaensenmannaertsantonastasozullo}.

\begin{df}
 Sets that can be constructed as in \Cref{ResHypercylinder} will be called \emph{hypercylinders}.
 The subspace $\tau$ in \Cref{ResHypercylinder} will be called the \emph{vertex} of the hypercylinder.
\end{df}

We note that if $q=2$, a hypercylinder is the complement of a hyperplane.
If $q \geq 4$ this is not true, and the vertex of the hypercylinder is the intersection of all hyperplanes disjoint to the hypercylinder, hence the vertex is uniquely determined by the hypercylinder.

\bigskip

It is a natural question to ask whether the only minimum sets of even type in $\pg(n,q)$ are the hypercylinders.
This is known to be true if $q=2$ \cite[Proposition 3]{bagchi}.
We will prove that it is also true for $q=4,8$.
As a first step, we reduce the problem to the case $n=3$.
We will use the following result.

\begin{res}[{\cite[Theorem 5.6]{adriaensenmannaertsantonastasozullo}}]
 \label{ResCasertaBoiz}
 Suppose that $S$ is a minimum set of even type in $\pg(n,q)$, $n \geq 4$.
 Assume that at least one plane intersects $S$ in a hyperoval and any solid through such a plane intersects $S$ in a hypercylinder.
 Then $S$ is a hypercylinder.
\end{res}

\begin{prop}
 \label{LmHypercylinderInduction}
Let $q$ be even.
Assume that the only minimum sets of even type in $\pg(3,q)$ are hypercylinders.
Then for any $n \geq 3$, the only minimum sets of even type in $\pg(n,q)$ are hypercylinders.
\end{prop}

\begin{proof}
We prove this proposition using \Cref{ResCasertaBoiz}.
Let $S$ be a minimum set of even type in $\PG(n,q)$, $n > 3$.
First, we prove that there exists a plane intersecting $S$ in a hyperoval.
Take a point $P \in S$.
Then there exists some line $\ell$ through $P$ with $|\ell \cap S| = 2$, otherwise every line through $P$ contains at least $3$ more points of $S$ implying $(q+2)q^{n-2} = |S| \geq 1 + 3 \theta_{n-3}$, a contradiction.
Likewise there exists some plane $\pi$ through $\ell$ containing only $q+2$ points of $S$, otherwise every plane through $\ell$ contains at least $q+4$ points of $S$ and $(q+2)q^{n-2} \geq 2 + \theta_{n-2}(q+2)$, again a contradiction.

Now suppose that $\pi$ is a plane intersecting $S$ in a hyperoval.
Then every solid through $\pi$ intersects $S$ in at least $(q+2)q$ points, since $S \cap \sigma$ is an even set in $\sigma$.
If some solid intersects $S$ in more than $(q+2)q$ points, then
\[
 |S| > (q+2) + \theta_{n-3}(q(q+2) - (q+2)) = (q+2)q^{n-2},
\]
which contradicts $S$ being a minimum set of even type.
By our hypothesis, this implies that every solid through $\pi$ intersects $S$ in a hypercylinder.
Thus, the conditions of \Cref{ResCasertaBoiz} are satisfied.
\end{proof}

Thus, we will focus our attention on minimum sets of even type in $\pg(3,q)$.
Sets of even type, or rather their complements, in $\pg(3,4)$ have been thoroughly examined, see e.g.\ \cite{hirschfeldhubaut}.
In particular, it follows from their classification that the only minimum sets of even type in $\pg(3,4)$ are hypercylinders, see \cite[Table 3]{hirschfeldhubaut}.
Hence the same holds in $\pg(n,4)$ for all $n$.
We note that for $n=4$, this was already proven in the PhD thesis of Packer \cite[Table 4.9]{packer95}.
Moreover, \cite[Table 4.11]{packer95} contains the weight distribution of $\mc_1(4,4)^\perp$.

\bigskip

Calkin, Key and de Resmini \cite[Proposition 3]{calkinkeyderesmini} proved that if $S$ a minimum set of even type in $\pg(3,q)$ that intersects every line in 0, 2, or $q$ points, then $S$ is a hypercylinder.
Our first goal is to relax this condition.

A set $B$ of points in $\pg(2,q)$ is called a \emph{blocking set} if no line of $\pg(2,q)$ is disjoint to $B$.
Blocking sets are well-studied objects, see e.g.\ \cite{blokhuissziklaiszonyi}.
Sets of even type and blocking sets are linked in the following way.

\begin{lm}
 \label{LmBlockingSet}
 Let $S$ be a set of even type in $\pg(3,q)$, let $\pi$ be a plane in $\pg(3,q)$, and let $\ell$ be a line in $\pi$.
 Then $(\pi \cap S) \triangle \ell$ is a blocking set in $\pi$.
\end{lm}

\begin{proof}
 Take a line $\ell'$ in $\pi$.
 Then
 \[
  \chi_{(\pi \cap S) \triangle \ell} \cdot \chi_{\ell'}
  = (\chi_{\pi \cap S} + \chi_\ell) \cdot \chi_{\ell'}
  = \chi_S \cdot \chi_{\ell'} + \chi_{\ell} \cdot \chi_{\ell'}
  = 0 + 1 = 1.
 \]
 Hence, $(\pi \cap S) \triangle \ell$ intersects $\ell'$ in an odd number of points, which in particular means that $(\pi \cap S) \triangle \ell$ and $\ell'$ are not disjoint.
\end{proof}

Suppose that $B$ is a blocking set and $\ell$ is a line not completely contained in $B$.
Take a point $P \in \ell \setminus B$.
There are $q$ more lines through $P$, all containing a point of $B$.
Hence, $|B \setminus \ell| \geq q$.
If equality holds, we say that $\ell$ is a \emph{Rédei line} of $B$, and $B$ is said to be \emph{of Rédei type}.
The above observation also implies the well-known fact \cite{boseburton} that a blocking set in $\pg(2,q)$ contains at least $q+1$ points, with equality if and only if the blocking set is a line.

\begin{df}
 Let $S$ be a minimum set of even type in $\pg(3,q)$.
 A plane $\pi$ is called a \emph{Rédei plane} with respect to $S$ if there exists a line $\ell \subset \pi$ such that $|S \cap \pi \setminus \ell| = q$, or equivalently $(S \cap \pi) \triangle \ell$ is a Rédei blocking set.
 We call $\ell$ a \emph{Rédei line} of $\pi$ with respect to $S$.
\end{df}

\begin{lm}
 \label{LmMax2q}
 Let $S$ be a minimum set of even type in $\pg(3,q)$, and let $\pi$ be a plane.
 Then $|S \cap \pi| \leq 2q$.
 If equality holds, then for every plane $\rho \neq \pi$, the line $\pi \cap \rho$ is either skew to $S$ or a Rédei line of $\rho$ with respect to $S$. 
\end{lm}

\begin{proof}
 Suppose that $\pi \cap S \neq \emptyset$.
 Take a line $\ell$ in $\pi$ containing some points of $S$.
 Then the $q$ planes $\rho \neq \pi$ through $\ell$ intersect $S$ in a non-empty set of even type .
 Hence, $(\rho \cap S) \triangle \ell$ is a blocking set by \Cref{LmBlockingSet}, and thus contains at least $q$ points outside of $\ell$, with equality if and only if $\ell$ is a Rédei line of $\rho$.
 Therefore
 \[
  q(q+2) = |S| \geq |S \cap \pi| + q \cdot q.
 \]
 This implies that $|S \cap \pi| \leq 2q$, and in case of equality, $\ell$ is a Rédei line in all the other planes through it.
\end{proof}

\begin{prop}
 \label{LmQSecantThenHypercylinder}
Suppose that $q$ is even, and that $S$ is a minimum set of even type in $\pg(3,q)$.
If there is a $q$-secant line to $S$, then $S$ is a hypercylinder.
\end{prop}

\begin{proof}
Suppose that $\ell$ is a $q$-secant line to $S$.
Let $P$ denote the unique point of $\ell \setminus S$.
Take a plane $\pi$ through $\ell$.
Since $(\pi \cap S) \triangle \ell$ is a blocking set, $\pi$ contains at least $q$ more points of $S$.
It have to be exactly $q$ points, otherwise $|S \cap \pi| > 2q$, contradicting \Cref{LmMax2q}.
Thus, $(\pi \cap S) \triangle \ell$ is a blocking set of size $q+1$, hence a line.
This means that $\pi \cap S$ is the symmetric difference of two lines through $P$.
Therefore, $S$ is the union of $q+2$ lines through $P$, without $P$.
Take a plane $\rho$ not containing $P$.
Then $\rho$ intersects each of these $q+2$ lines in a different point.
Therefore, $S \cap \rho$ contains $q+2$ points, hence it is a hyperoval.
It now follows directly that $S$ is a hypercylinder with vertex $P$.
\end{proof}

Given a minimum set of even type $S$, call a line $\ell$ a \emph{large secant} if $|\ell \cap S| > \frac q 2$.
We will exploit the theory of Rédei blocking sets to prove that if $S$ is not a hypercylinder, it has at most one large secant, and a large secant implies the existence of a non-trivial subfield of $\FF_q$.
We will make use of the celebrated theorem by Ball, Blokhuis, Brouwer, Storme, and Sz\H onyi.

\begin{res}[\cite{blokhuisballbrouwerstormeszonyi,ball03}]
 \label{ResRedeiBall}
Let $B$ be a Rédei blocking set of $\pg(2,q)$ of size $q+m$, $q = p^h$, $p$ prime.
Let $s = p^e$ be maximal such that every line intersects $B$ in $1 \pmod s$ points. Then one of the following holds:
\begin{enumerate}
 \item $s=1$ and $\frac{q+3}2 \leq m \leq q+1$,
 \item $\FF_s$ is a proper subfield of $\FF_q$, and $\frac q s + 1 \leq m \leq \frac{q-1}{s-1}$,
 \item $s=q$, and $B$ is a line.
\end{enumerate}
\end{res}

\begin{lm}
 \label{LmLargeSecant}
Suppose that $S$ is a minimum set of even type in $\pg(3,q)$, and $S$ is not a hypercylinder.
Then there is at most one large secant line to $S$.
If a large secant $\ell$ exists, there exists some subfield $\FF_s$ of $\FF_q$, with $s \neq 2,q$, such that $q+1-\frac{q-1}{s-1} \leq |\ell \cap S| \leq q - \frac q s$.
\end{lm}

\begin{proof}
Suppose that $\ell$ is a large secant, and $|\ell \cap S| = m$.
Then there exists a Rédei plane $\pi$ through $\ell$,
otherwise, all planes through $\ell$ contain at least $q+2$ points of $S \setminus \ell$ and
\begin{equation}
 \label{EqExistsRedei}
 q(q+2) = |S| \geq m + (q+1)(q+2),
\end{equation}
a contradiction.
Let $B$ be the Rédei blocking set $(\pi \cap S) \triangle \ell$.
Then $k := |\ell \cap B| = (q+1) - m \leq \frac q 2 - 1$.
By Result \ref{ResRedeiBall}, there exists some subfield $\FF_s$ such that $\frac q s + 1 \leq k \leq \frac{q-1}{s-1}$ or $s=q$ and $B$ is a line.
But if $B$ were a line, then $\ell$ would be a $q$-secant, and by \Cref{LmQSecantThenHypercylinder}, $S$ would be a hypercylinder.
Thus, $q+1-\frac{q-1}{s-1} \leq m \leq q - \frac q s$.
Since $m > \frac q 2$, $s \neq 2$.
In particular, this implies that $m \geq q+1 - \frac{q-1}{4-1} = \frac{2q + 4}3$.

Next, we prove that there cannot be more than one large secant.
So suppose that there are two large secants $\ell_1$ and $\ell_2$.
Write $m_i = |S \cap \ell_i|$.
We distinguish two cases.

\noindent \underline{Case 1. $\ell_1$ and $\ell_2$ intersect in a point $P$.}

Consider the plane $\pi$ spanned by $\ell_1$ and $\ell_2$.
Take a point $Q \in \ell_2 \setminus (S \cup \set P)$.
There are at least $m_1-1$ lines through $Q$, different from $\ell_2$ that intersect $\ell_1$ in a point of $S$.
Each of these lines must contain an extra point of $S$, not contained on $\ell_1$ or $\ell_2$.
Then
\[
 |S \cap \pi| \geq m_2 + 2 (m_1 -1) \geq 3\frac{2q+4}3 - 2 > 2q,
\]
which contradicts \Cref{LmMax2q}.

\noindent \underline{Case 2. $\ell_1$ and $\ell_2$ don't intersect.}

Take a point $P \in \ell_1 \cap S$.
If no line through $P$ in $\vspan{P,\ell_2}$ is a 2-secant, then all these lines contain at least 3 extra points of $S$, implying $|S \cap \vspan{P,\ell_2}| \geq 1+3(q+1)$, contradicting \Cref{LmMax2q}.
So let $\ell$ be a 2-secant through $P$ and some point of $\ell_2$.
Then for $i=1,2$, the plane $\vspan{\ell,\ell_i}$ contains at least $m_i+q$ points of $S$.
All other planes contain at least $q$ extra points of $S$.
Thus, looking at the planes through $\ell$,
\[
 q(q+2) = |S| \geq 2 + 2\left(\frac{2q+4}3 + q - 2 \right) + (q-1)q = (q+2)q + \frac{q+2}3,
\]
a contradiction.
\end{proof}

Unfortunately, the above proposition does not seem to be strong enough a tool to prove that the only minimum sets of even type in $\pg(3,q)$ are hypercylinders for general values of $q$.
However, if $q=8$, the above lemma proves that a minimum set of even type cannot have $6$-secant lines, and we will show below that this suffices to prove that hypercylinders are the only minimum sets of even type.
We remark that the exclusion of $6$-secant lines can be obtained using less heavy machinery.
Indeed, if a 6-secant $\ell$ would exist, it would lie in a Rédei plane (cf.\ \Cref{EqExistsRedei}).
This would imply the existence of a blocking set in $\PG(2,q)$ not containing a line, of size $((8+1)-6) + 8 = 11 < 8 + \sqrt 8 + 1$.
By a celebrated result of Bruen \cite{bruen71}, this is not possible.

\begin{prop}
 The only minimum sets of even type in $\pg(3,8)$ are hypercylinders.
\end{prop}

\begin{proof}
 Suppose to the contrary that $S$ is a set of even type of size $8 \cdot (8+2) = 80$, not a hypercylinder.
 Since $\FF_8$ has no proper subfield but $\FF_2$, \Cref{LmLargeSecant} implies that $S$ has no large secants.
 Therefore, every line intersects $S$ in 0, 2, or 4 points.
 
 Take a point $P \in S$.
 Let $x$ denote the number of $4$-secant lines through $P$.
 Then
 \[
  80 = |S| = 1 + (\theta_2-x)1 + 3x = 74 + 2x,
 \]
 hence $x = 3$.
 
 For every point $P$ in $S$ there are two options:
 \begin{enumerate}
  \item The three 4-secant lines through $P$ lie in a plane $\pi$.
  In this case, $|\pi \cap S| = 16$.
  Since each point of $\pi \cap S$ lies only on 2- and 4-secant lines, its three 4-secant lines must be contained in $\pi$.
  Moreover, there is no plane through $P$ containing exactly two 4-secant lines through $P$, hence $P$ does not lie on a $14$-secant plane.
  \item The three 4-secant lines through $P$ do not lie in a plane.
  In this case there are three planes through $P$ which contain two 4-secant lines and are 14-secant.
 \end{enumerate}
 Let $m_i$ denote the number of $i$-secant planes.
 Since every point lies in at most one 16-secant plane, the number of such points equals $16 m_{16}$.
 Now we do a double count on the incident point-plane tuples $(P,\pi)$ with $P \in S$ and $\pi$ a 14-secant plane.
 This number equals
 \[
  14 m_{14} = (80-16m_{16}) 3 = 16(5-m_{16})3.
 \]
 If $m_{14} \neq 0$, then $16(5-m_{16})3$ divides 14, which is impossible, since all its factors are coprime with 7.
 Thus, $m_{16} = 5$ and $m_{14}=0$.
 
 Hence, every point of $S$ lies on a unique 16-secant plane.
 Denote these planes as $\rho_1, \dots, \rho_5$.
 We say that a point $P \in S$ is of type $i$ if $P \in \rho_i$.
 
 Take a point $P_1$ of type 1 and a point $P_2$ of type 2.
 Consider the line $\vspan{P_1,P_2}$.
 Since it contains points of different types, it can not lie on $16$-secant planes.
 On the other hand, it contains points of $S$, so it can not lie on 0-secant planes.
 Therefore, $\vspan{P_1,P_2}$ only lies on 10- and 12-secant planes.
 A plane $\pi$ through $\vspan{P_1,P_2}$ is 12-secant if and only if it contains a unique $4$-secant through $P_1$ if and only if intersects $\rho_1$ in a 4-secant line.
 Similarly, $\pi$ is a 12-secant plane, if and only if $\pi$ intersects $\rho_2$ in a 4-secant line through $P_2$.
 Thus, the 4-secants through $P_1$ intersect the line $\ell = \rho_1 \cap \rho_2$ in the same points as the 4-secants through $P_2$.
 Let $R_1, R_2, R_3$ denote the points of intersection of $\ell$ with the 4-secant lines through $P_1$.
 If we fix one of the points in $\{P_1,P_2\}$ and vary the other, we see that all 4-secants in $\rho_1$ and $\rho_2$ intersect $\ell$ in $R_1$, $R_2$, or $R_3$.
 
 Now we prove that these three points are the only points in $\rho_1 \setminus S$ on more than one 4-secant line of $\rho_1$.
 Indeed, there are $\frac{16\cdot 3}4 = 12$ 4-secant lines in $\rho_1$, since any of the 16 points of $\rho_1 \cap S$ lies on 3 such lines, and any such line contains 4 points of $\rho_1 \cap S$.
 Hence, the size of
 \[
  V = \sett{(\ell_1, \ell_2, Q)}{\ell_1, \ell_2 \textnormal{ distinct 4-secant lines in } \rho_1, \, Q = \ell_1 \cap \ell_2 }
 \]
 equals $12 \cdot 11$.
 On the other hand every point of $S$ lies on three 4-secants of $\rho_1$, so accounts for $3 \cdot 2 = 6$ triples of $V$, and the points $R_i$ lie on four 4-secants of $\rho_1$, so account for $4 \cdot 3 = 12$ triples in $V$.
 In total this accounts for $16 \cdot 6 + 3 \cdot 12 = 12 \cdot 11$ elements of $V$, hence all elements of $V$.
 
 We can repeat the above argument with $P_2$ replaced by a point $P_i \in \rho_i$ with $i \in \{3,4,5\}$.
 Then we find three points $R_1', R_2', R_3'$ of $\rho_1 \cap \rho_i$ such that every 4-secant line of $\rho_1$ and $\rho_i$ intersects $\rho_1 \cap \rho_i$ in one of these points.
 Since they are points of $\rho_1 \setminus S$ one more than one 4-secant of $\rho_1$, $\{R_1,R_2,R_3\} = \{ R_1', R_2', R_3' \}$.
 We conclude that all 16-secant planes go through $\ell$, and all 4-secant lines intersect $\ell$ in $R_1$, $R_2$, or $R_3$.
 
 Thus, $\ell$ lies on five 16-secant planes, and the other planes through $\ell$ must be 0-secant.
 None of these planes contain 2-secant lines through $R_1$.
 In particular, this implies that $R_1$ does not lie on any 10-secant planes, since such a plane $\pi$ intersects $S$ in a hyperoval, and contains 2-secant lines through every point of $\pi$.
 Let $x$ denote the number of 12-secant planes through $R_1$.
 By performing a double count on
 \[
  \sett{(P,\pi)}{P \in S, \, \pi \textnormal{ a plane, } P,R_1 \in \pi },
 \]
 we see that
 \[
  x \cdot 12 + 5 \cdot 16 = |S| 9 = 80 \cdot 9.
 \]
 But then $x$ is not integer, a contradiction that concludes the proof.
\end{proof}

As we mentioned before, the only minimum sets of even type in $\pg(3,4)$ are the hypercylinders, and the only hyperovals in $\PG(2,q)$, $q \in \{4,8\}$, are regular.
This together with the above proposition and \Cref{LmHypercylinderInduction} proves \Cref{ThmMainHypercylinders}, stating that the only minimum sets of even type in $\pg(3,q)$, with $q \in \{4,8\}$, are hypercylinders over regular hyperovals.
Using \Cref{ResReductionToLines}, this classifies the minimum weight codewords of the codes $\mc_k(n,q)^\perp$ for $q \in \{4,8\}$.

\bigskip

We can now determine the number of minimum weight codewords of $\mc_k(n,q)^\perp$ for $q \in \{4,8\}$.
We need some basic facts about hyperovals, which can e.g.\ be found in \cite[Table 7.2, \S 8]{hirschfeld98}.
A regular hyperoval consists of an irreducible conic together with its nucleus.
There are $q^5-q^2$ irreducible conics in $\pg(2,q)$.
Any 5 points in $\pg(2,q)$, no 3 of which are collinear, lie on a unique irreducible conic.

First consider the case $q=4$.
Let $\mo$ be a hyperoval in $\pg(2,4)$.
If you delete any point from $\mo$, you are left with 5 points, hence they constitute an irreducible conic.
Therefore, the total number of hyperovals in $\pg(2,4)$ equals $\frac{4^5-4^2}{6} = 168$.

If $q=8$, then distinct irreducible conics give rise to distinct hyperovals.
Otherwise, they would intersect in $q>4$ points, which yields a contradiction.
Therefore, there are $8^5 - 8^2 = 32 \, 704$ hyperovals in $\pg(2,8)$.

Lastly, we use the fact the number of $k$-spaces in $\pg(n,q)$ is given by the Gaussian coefficient, see e.g.\ \cite[Theorem 3.1]{hirschfeld98}.
\[
 \gauss{n+1}{k+1} = \prod_{i=0}^k \frac{q^{n+1-i}-1}{q^{k+1-i}-1}.
\]

\begin{crl}
 Suppose that $q \in \{4,8\}$.
 The minimum weight of $\mc_k(n,q)^\perp$ is $q^{n-2}(q+2)$, and the minimum weight codewords of $\mc_k(n,q)^\perp$ are exactly the characteristic vectors of the hypercylinders embedded in an $(n-k+1)$-space.
 Define the number
 \[
  \delta(q) = \begin{cases}
   168 & \text{if } q = 4, \\
   32 \, 704 & \text{if } q=8.
  \end{cases}
 \]
 The number of minimum weight codewords of $\mc_k(n,q)^\perp$ is given by
 \[
  \gauss{n+1}{k-1} \gauss{n-k+2}{3} \delta(q),
 \]
 and all minimum weight codewords are equivalent under the action of the automorphism group of the code.
\end{crl}

\begin{proof}
 By \Cref{ResReductionToLines} and \Cref{ThmMainHypercylinders}, a minimum weight codeword of $\mc_k(n,q)^\perp$ is the characteristic function of a hypercylinder embedded in an $(n-k+1)$-space.
 Since this hypercylinder spans the $(n-k+1)$-space, the total number of minimum weight codewords equals the number of $(n-k+1)$-spaces, i.e.\ $\gauss{n+1}{n-k+2} = \gauss{n+1}{k-1}$, multiplied by the number of hypercylinders in $\pg(n-k+1,q)$.
 As we mentioned above, the vertex of a hypercylinder is uniquely determined by the hypercylinder whenever $q>2$, so there are $\gauss{n-k+2}{(n-k+2)-3} = \gauss{n-k+2}3$ choices for a vertex, which all give rise to distinct hypercylinders.
 Once the vertex is fixed, the hypercylinder is constructed by choosing a hyperoval in a plane disjoint to the vertex.
 The choice of the plane is irrelevant, all planes disjoint to the vertex yield the same hypercylinders.
 Thus, the number of hypercylinders in $\pg(n-k+1,q)$ with a fixed vertex equals the number $\delta(q)$ of hyperovals in $\pg(2,q)$.
 This proves the formula for the number of minimum weight codewords.

 Lastly, every collineation of $\pg(n,q)$ gives rise to an automorphism of $\mc_k(n,q)^\perp$.
 Since all hyperovals in $\pg(2,q)$ are projectively equivalent, it easily follows that all hypercylinders embedded in an $(n-k+1)$-space are projectively equivalent.
 Thus, their characteristic vectors are in the same orbit of the automorphism group of $\mc_k(n,q)^\perp$.
\end{proof}

\section{Small weight codewords in the code of points and hyperplanes}
 \label{SecSmallWtPtHyp}

In this section, we will give a short and self-contained proof of \Cref{ResOlgaFerdi}.
The main tool is \Cref{Lm:BBWZ}.
This lemma generalises a powerful lemma by Blokhuis, Brouwer, and Wilbrink \cite[Proposition]{blokhuisbrouwerwilbrink} that works for $\mc_1(2,q)$.
This generalisation was presented to the author by Olga Polverino and Ferdinando Zullo, although they never published it.

\begin{df}
For a set of points $S$ in $\PG(n,q)$ and a point $P$ (which can be in or outside of $S$), the \emph{feet} of $P$ with respect to $S$ are the points $R \in S$ such that
\[
 \vspan{P, R} \cap S \setminus \set P = \set R.
\]
\end{df}

\begin{lm}
 \label{Lm:BBWZ}
 Suppose that $c \in \mc_{n-1}(n,q)$, $q>2$, and $P$ is a point of $\PG(n,q)$.
 Then the feet of $P$ with respect to $\supp(c)$ span a subspace that does not contain $P$.
\end{lm}

\begin{proof}
 Suppose that $P$ is contained in the subspace spanned by its feet.
 Let $R_0, \dots, R_k$ be a set of feet of $P$ of smallest possible size such that $P \in \vspan{R_0, \dots, R_k}$.
 By the minimality, $P, R_0, \dots, R_k$ constitute a frame of $\vspan{R_0,\dots,R_k}$, and there exist linearly independent vectors $e_0, \dots, e_k \in \FF_q^{n+1}$ such that
 \begin{align*}
  P = \vspan{e_0}, &&
  R_1 = \vspan{e_1}, &&
  \dots, &&
  R_k = \vspan{e_k}, &&
  R_0 = \vspan{e_0 + e_1 + \ldots + e_k}.
 \end{align*}
 Let $C'$ denote the span of $\mc_{n-1}(n,q)$ over $\FF_q$.
 Then $\mc_{n-1}(n,q) \subseteq C'$, so $c \in C'$.
 We will now construct a vector in $C'^\perp$.
 Let $\mp$ denote the set of points of $\pg(n,q)$ and consider
 \[
  v: \mp \to \FF_q:
  Q \mapsto \begin{cases}
   \alpha & \text{if $Q = \vspan{\alpha e_0 + e_i}$ with } 1 \leq i \leq k, \\
   - \alpha & \text{if } Q = \vspan{\alpha e_0 + e_1 + \ldots + e_k}, \\
   0 & \text{otherwise}.
  \end{cases}
 \]
 For $i = 0, \dots, k$ denote $\ell_i = \vspan{P,R_i}$.
 Then $\supp(v) = \left(\bigcup_{i=0}^k \ell_i \right) \setminus \set{P,R_0',R_1 \dots, R_k}$, with $R_0' = \vspan{e_1 + \ldots + e_k}$.
 Let us check that $v \in C'^\perp$.
 Choose a hyperplane $\pi$.
 If $P \in \pi$, then every line $\ell_i$ either intersects $\pi$ in $P$ or is contained in $\pi$.
 Since $v(P) = 0$ and $\sum_{Q \in \ell_i} v(Q) = \sum_{\alpha \in \FF_q} \alpha = 0$ (because $q>2$), $v \cdot \chi_\pi = 0$.
 If $P \notin \pi$, then for $i=1,\dots,k$, $\pi$ intersects $\ell_i$ in some point $\vspan{\alpha_i e_0 + e_i}$, and it must intersect $\ell_0$ in $\vspan{(\alpha_1 + \ldots + \alpha_k)e_0 + e_1 + \ldots + e_k}$.
 Therefore,
 \[
  v \cdot \chi_\pi = \alpha_1 + \ldots + \alpha_k - (\alpha_1 + \ldots + \alpha_k) = 0.
 \]
 Hence, indeed $v \in C'^\perp$.

 Since for each line $\ell_i$, $\supp(c) \cap \ell_i \setminus \set P = \set{R_i}$ and $\supp(v) = \left(\bigcup_{i=0}^k \ell_i \right) \setminus \set{P,R_0',R_1 \dots, R_k}$,
 $\supp(c) \cap \supp(v) = \set{R_0}$.
 Thus, $v \cdot c = v(R_0) c(R_0) = - c(R_0) \neq 0$.
 This contradicts that $c \in C'$ and $v \in C'^\perp$.
 Therefore, we may conclude that $P$ is not contained in the subspace spanned by its feet with respect to $\supp(c)$.
\end{proof}

We will use one more simple lemma, and include a proof for the sake of completeness.

\begin{lm}[{\cite[Lemma 2]{lavrauw}}]
 \label{LmConstInprod}
 Suppose that $c \in \mc_{n-1}(n,q)$.
 Then there exists a scalar $\beta \in \FF_p$ such that for each subspace $\rho$ of dimension at least 1, $c \cdot \chi_\rho = \beta$.
\end{lm}

\begin{proof}
 Since $c \in \mc_{n-1}(n,q)$, $c = \sum_i \alpha_i \chi_{\pi_i}$ for some scalars $\alpha_i \in \FF_p$ and hyperplanes $\pi_i$.
 Define $\beta = \sum_i \alpha_i$.
 Take a subspace $\rho$ of dimension at least 1.
 For every hyperplane $\pi_i$, $\dim \rho \cap \pi_i \geq 0$.
 This implies that $|\rho \cap \pi_i| \equiv 1 \pmod p$ and $\chi_\rho \cdot \chi_{\pi_i} = 1$.
 Hence,
 \[
  c \cdot \chi_\rho
  = \left( \sum_i \alpha_i \chi_{\pi_i} \right) \cdot \chi_\rho 
  = \sum_i \alpha_i (\chi_{\pi_i} \cdot \chi_\rho)
  = \sum_i \alpha_i = \beta. \qedhere
 \]
\end{proof}

The following proof also uses the notion of an \emph{arc} in $\pg(2,q)$.
This is a set of points in $\pg(2,q)$, no three on a line.
It is well-known that an arc contains at most $q+2$ points, and that there are always lines which are disjoint to the arc, see e.g.\ \cite[\S 8.1]{hirschfeld98}.

\begin{thm}
 Let $q>2$.
 Suppose that $c \in \mc_{n-1}(n,q)$ and $0 < \wt(c) \leq 2q^{n-1}$.
 Then either
 \begin{itemize}
  \item $\wt(c) = \theta_{n-1}$ and $c = \alpha \chi_\pi$,
  \item $\wt(c) = 2q^{n-1}$ and $c = \alpha (\chi_\pi - \chi_\rho)$,
 \end{itemize}
 with $\alpha \in \FF_p$, and $\pi$, $\rho$ hyperplanes of $\pg(n,q)$.
\end{thm}

\begin{proof}
  Take a codeword $c \in \mc_{n-1}(n,q)$, and let $\beta$ be as in \Cref{LmConstInprod}.

 \noindent \underline{Case 1: $\beta \neq 0$}.

 Define
 \[
  m = \max \sett{|\supp(c) \cap \pi|}{\pi \textnormal{ a hyperplane of } \pg(n,q)}
 \]
 as the maximum number of points of $\supp(c)$ in a hyperplane.
 We will show that $m = \theta_{n-1}$.
 Suppose to the contrary that $m < \theta_{n-1}$.
 Take an $m$-secant hyperplane $\pi$ and a point $P \in \pi \setminus \supp(c)$.
 Note that every line contains a point of $\supp(c)$ since  $\beta \neq 0$.
 Let $x$ denote the number of tangent lines through $P$.
 Since there are $q^{n-1}$ lines through $P$ outside of $\pi$,
 \begin{equation}
  \label{eq:Beta=/=0}
  2q^{n-1} \geq \wt(c) \geq m + x \cdot 1 + (q^{n-1}-x) \cdot 2 = 2q^{n-1} + m-x.
 \end{equation}
 This implies that $x \geq m$.
 On the other hand, by \Cref{Lm:BBWZ}, there exists some hyperplane $\tau \not \ni P$ containing the $x$ feet of $P$ with respect to $\supp(c)$.
 Therefore, $x \leq m$.
 Hence, $x=m$, and we have equality in \Cref{eq:Beta=/=0}.
 Furthermore, $\tau$ must intersect $\supp(c)$ exactly in the $m$ feet of $P$, and equality in \Cref{eq:Beta=/=0} can only hold if none of the feet of $P$ are contained in $\pi$.
 Hence, $\pi \cap \tau \cap \supp(c) = \emptyset$.
 If $n\geq 3$, then $\dim(\pi \cap \tau) \geq 1$, so $\chi_{\pi \cap \tau} \cdot c = \beta \neq 0$, a contradiction.
 If $n=2$, then $m > 2$, otherwise $\supp(c)$ is an arc and there are lines missing $\supp(c)$ which would imply that $\beta = 0$.
 By equality in \Cref{eq:Beta=/=0}, $P$ lies on a unique $m$-secant line, $m$ tangent lines, and $q-m$ 2-secant lines.
 Moreover, this holds for any point outside $\supp(c)$, lying on an $m$-secant line.
 But then $\pi \cap \tau$ would be a point outside of $\supp(c)$ lying on at least 2 $m$-secant lines, a contradiction.
 So in both cases, $m < \theta_{n-1}$ leads to a contradiction.

 Thus, $\supp(c)$ contains a hyperplane $\pi$.
 Take a point $P \in \pi$.
 Then $c(P) = \beta$.
 Otherwise, every line through $P$ outside of $\pi$ contains at least one other point of $\supp(c)$ and $\wt(c) \geq \theta_{n-1} + q^{n-1} > 2q^{n-1}$.
 Furthermore, if $\supp(c)$ contains a point $P \notin \pi$, then every line $\ell$ through $P$ contains another point of $\supp(c) \setminus \pi$.
 Otherwise, $\chi_\ell \cdot c = c(P) + \beta \neq \beta$, a contradiction.
 But then $\wt(c) \geq 1 + 2 \theta_{n-1}$, a contradiction.
 Hence, $\supp(c) = \pi$ which implies that $c = \beta \chi_\pi$.

 \noindent \underline{Case 2: $\beta = 0$.}

 Let $A = \sett{c(P)}{P \in \supp(c)}$ denote the set of non-zero coefficients that $c$ takes.
 For every $\alpha \in A$ define
 \[
  m_\alpha = \max \sett{|\supp(c) \cap \pi|}{\pi \textnormal{ a hyperplane of } \pg(n,q), \, (\exists P \in \pi)(c(P) = \alpha)}.
 \]
 Now suppose that $\alpha \in A$, and $m_{\alpha} \geq m_{-\alpha}$ (where we set $m_{-\alpha}=0$ if $-\alpha \notin A$).
 Take an $m_\alpha$-secant hyperplane $\pi$ containing a point $P$ with $c(P) = \alpha$.
 Let $x$ denote the number of 2-secant lines through $P$.
 Since there are no tangent lines to $\supp(c)$,
 \begin{equation}
  \label{eq:Beta=0}
  2q^{n-1} \geq \wt(c) \geq m_\alpha + x \cdot 1 + (q^{n-1}-x) \cdot 2
  = 2q^{n-1} + m_{\alpha} - x.
 \end{equation}
 Thus, $x \geq m_\alpha$.
 On the other hand, the $x$ feet of $P$ lie in some hyperplane $\tau \not \ni P$, hence $x \leq m_{-\alpha} \leq m_\alpha$.
 Therefore, $x = m_\alpha = m_{-\alpha}$.
 This implies that equality holds in \Cref{eq:Beta=0}.
 Therefore, the points of $\tau \cap \supp(c)$ are exactly the $m_\alpha$ feet of $P$, $\pi \cap \tau \cap \supp(c) = \emptyset$, and the lines $\ell$ with $P \in \ell \not \subseteq \pi$ consist of $m_\alpha$ 2-secants, and $q^{n-1}-m_\alpha$ 3-secants.

 If $q$ is even, then there are no 3-secant lines $\ell$ to $\supp(c)$, because otherwise $\beta = \chi_\ell \cdot c = 1$, contradicting our assumption that $\beta = 0$.
 In this case, we find $m_1 = q^{n-1}$.
 Then $\pi$ and $\tau$ are $q^{n-1}$-secant hyperplanes and their intersection is empty, hence $\supp(c) = \pi \triangle \rho$ and $c = \chi_\pi - \chi_\rho$.

 So for the rest of the proof assume that $q$ is odd.
 Since no 3-secant line through $P$ can contain a point with coefficient $-\alpha$, all points with coefficient $-\alpha$ are contained in $\pi \cup \tau$.
 Since we proved that $m_\alpha = m_{-\alpha}$, we can repeat the previous argument where we replace $P$ by some point $R$ of $\tau$ with coefficient $-\alpha$.
 Then we find some $m_\alpha$-secant hyperplane $\pi'$ such that all points with coefficient $\alpha$ are contained in $\tau \cup \pi'$.
 We proved that all points of $\tau \cap \supp(c)$ have coefficient $-\alpha$, hence $\pi'$ is an $m_\alpha$-secant hyperplane that intersects $\supp(c)$ exactly in the points with coefficient $\alpha$.
 In conclusion, for every scalar $\alpha \in A$, there exists an $m_\alpha$-secant hyperplane, which we'll denote by $\pi_\alpha$, such that $\pi_\alpha \cap \supp(c) = \sett{P \in \supp(c)}{c(P) = \alpha}$.

 Next, we prove that $A = \set{\pm \alpha}$ for some scalar $\alpha$.
 Take a scalar $\alpha \in A$, then we know that $-\alpha \in A$.
 Suppose that $A$ also contains a scalar $\gamma \neq \pm \alpha$.
 Take a point $P$ with $c(P) = \gamma$.
 Then $P \notin \pi_\alpha \cup \pi_{-\alpha}$, so every line through $P$ and a point of $\pi_\alpha \cap \supp(c)$ is a 3-secant.
 Therefore, it contains a point with coefficient $-\alpha-\gamma$.
 Hence, $m_{\alpha+\gamma} = m_{-\alpha-\gamma} \geq m_\alpha$ for every $\alpha, \gamma \in A$.
 It follows that $A = \FF_p^*$ and $m_\alpha$ is equal for all $\alpha$ in $A$.
 But then $2q^{n-1} = \wt(c) = |A|m_\alpha = (p-1)m_\alpha$.
 Since $p-1$ is coprime with $q^{n-1}$, $p-1$ must divide 2.
 This implies that $p=3$, in which case $A = \FF_3^* = \set{\pm1}$.

 Hence, $A = \set{\pm \alpha}$, and $2q^{n-1} = \wt(c) = 2m_\alpha$.
 This implies that $m_\alpha = q^{n-1}$.
 Since $\pi_\alpha \cap \pi_{-\alpha} \cap \supp(c) = \emptyset$, the only option is that
 \[
  c(P) = \begin{cases}
      \alpha & \text{if } P \in \pi_\alpha \setminus \pi_{-\alpha}, \\
      -\alpha & \text{if } P \in \pi_{-\alpha} \setminus \pi_\alpha, \\
      0 & \text{otherwise.}
  \end{cases}
 \]
 I.e.\ $c = \alpha (\chi_{\pi_\alpha} - \chi_{\pi_{-\alpha}})$.
\end{proof}

This proves \Cref{ResOlgaFerdi} in case $q>2$.
For the sake of completeness, we also include a proof for $q=2$.

\begin{lm}
 The only codewords in $\mc_{n-1}(n,2)$ are $\zero$, $\one$, and the codewords of the form $\chi_\pi$ and $\chi_\pi - \chi_\rho$, with $\pi$ and $\rho$ hyperplanes.
\end{lm}

\begin{proof}
 It suffices to prove that if we take a codeword $c$ of one of the forms described in the lemma, then for any hyperplane $\pi$, $c + \chi_\pi$ is also of one of these forms.
 This clearly holds if $c = \zero$ or $c = \chi_\rho$ for a hyperplane $\rho$.
 Note that if $\rho$ and $\sigma$ are hyperplanes $\chi_\rho - \chi_\sigma = \chi_{\rho \triangle \sigma}$ and $\rho \triangle \sigma$ is the complement of the unique hyperplane $\tau$ distinct from $\rho$ and $\sigma$ through $\rho \cap \sigma$.
 Hence, $\chi_\rho - \chi_\sigma = \one - \chi_\tau$.
 It immediately follows that if $c = \chi_\rho - \chi_\sigma$ or $c = \one$, then $c + \chi_\pi$ has the form of one of the codewords described in the lemma.
\end{proof}

\section{Conclusion}
 \label{SecConclusion}

In this paper, we studied small weight codewords of $\mc_k(n,q)$ and $\mc_k(n,q)^\perp$.
Many questions remain, especially regarding the dual codes.

\begin{prob}
 Determine for $q=2^h > 8$ whether or not there are minimum even sets in $\PG(3,q)$ different from the hypercylinders.
\end{prob}

We note that \Cref{ThmMainHypercylinders} cannot hold for $q > 8$ without removing the adjective ``regular'', since every Desarguesian projective plane of even order $q > 8$ contains non-regular hyperovals, and they have been classified for $q  \leq 64$ \cite{vandendriessche19}.

\begin{prob}
 \label{ProbMinWeightDualPlane}
 For $q$ odd and non-prime, determine the minimum weight of $\mc_1(2,q)^\perp$, and characterise the minimum weight codewords.
\end{prob}

For general $q$, as far as the author knows, the best known lower bound on the minimum weight is given by $2(q - \frac q p + 1)$ (see \Cref{CrlMinWeight}), and the best known upper bound is given $2q - \frac{q-1}{p-1} + 1$ \cite{key}.
As we stated before, the lower bound has recently been improved to $2q - 2\frac q p + 5$ in case $q = p^2 \geq 5^2$ \cite{deboeckvandevoorde2022}.

\Cref{ProbMinWeightDualPlane} also remains open in $\PG(n,q)$ with $n>2$.

\begin{prob}
 Do all minimum weight codewords of $\mc_1(n,q)^\perp$ arise from minimum weight codewords of $\mc_1(2,q)^\perp$?
\end{prob}

A general construction, of which hypercylinders are an instance, to construct small weight codewords of $\mc_1(n,q)^\perp$ from small weight codewords in $\mc_1(2,q)^\perp$ is given by \cite[Lemma 6]{bagchi} or \cite[Construction 7.12]{adriaensendenaux}.
For $q$ prime, all minimum weight codewords are instances of this construction \cite[Proposition 2]{bagchi}, and now we know that the same holds for $q \in \{4,8\}$.

\bigskip

\noindent \textbf{Acknowledgements.}
The author would like to thank Aart Blokhuis and Zsuzsa Weiner for the communication concerning the proof of \Cref{ResLowerBound} (1), and Olga Polverino and Ferdinando Zullo for sharing \Cref{Lm:BBWZ} and helpful discussions.

\newcommand{\etalchar}[1]{$^{#1}$}

\end{document}